\documentclass{amsart}

\usepackage{epsfig,amsmath,xypic}
\usepackage[psamsfonts]{amssymb}
\usepackage[latin1]{inputenc}
\usepackage{pstricks,pst-node,pst-plot}
\usepackage{hyperref}
\usepackage{diagbox}
\usepackage{enumerate}

\usepackage[mathscr]{euscript}
\usepackage{calrsfs}
\DeclareMathAlphabet{\pazocal}{OMS}{zplm}{m}{n}

\newcommand{\ob}{{\hspace{.2em}[\hspace{-.4em}[\hspace{.2em}}}
\newcommand{\cb}{{\hspace{.2em}]\hspace{-.4em}]\hspace{.2em}}}

\newtheorem{theorem}{\bf Theorem}
\newtheorem{proposition}[theorem]{\bf Proposition}
\newtheorem{lemma}[theorem]{\bf Lemma}
\newtheorem{corollary}[theorem]{\bf Corollary}
\newtheorem*{conjecture}{\bf Conjecture}

\theoremstyle{remark}

\newtheorem*{remark}{\bf Remark}

\def\cal{\pazocal}

\def\and{{\quad\text{and}\quad}}
\def\with{{\quad\text{with}\quad}}

\def\N{{\mathbb N}}
\def\Z{{\mathbb Z}}
\def\Q{{\mathbb Q}}

\def\C{{\mathbb C}}

\def\F{{\mathbb F}}

\def\P{{\mathbb P}}

\def\result{{\rm resultant}}

\def\deg{{\rm deg}}

\def\mod#1{{\ ({\rm mod}\ #1)}}

\def\Sik{{\Sigma_k}}
\def\Sk{{\cal S}_k}
\def\Ck{{\cal V}_k}

\def\L{{\cal L}}

\def\Gal{{\rm Gal}(\overline\Q/\Q)}

\def\ssm{{\smallsetminus}}

\subjclass{}

\begin{author}[X.~Buff]{Xavier Buff}\thanks{The research of the first author was supported in part by the ANR grant Lambda ANR-13-BS01-0002}
\email{xavier.buff@math.univ-toulouse.fr}
\address{ %
  Institut de Mathématiques de Toulouse\\
   UMR5219\\ Université de Toulouse, CNRS, UPS\\ F-31062 Toulouse Cedex 9\\ France }
\end{author}

\begin{author}[A.L.~Epstein]{Adam L. Epstein}
\email{adame@maths.warwick.ac.uk}
\address{ %
Mathematics institute\\
 University of Warwick\\
 Coventry CV4 7AL\\
 United Kingdom }
\end{author}

\begin{author}[S.~Koch]{Sarah Koch}\thanks{The research of the third author was supported in part by the NSF}
\email{kochsc@umich.edu}
\address{Department of Mathematics\\
530 Church Street\\
East Hall\\
University of Michigan\\
Ann Arbor MI 48109\\
United States }
\end{author}

\begin{document} 
 
 \title{Rational maps with a preperiodic critical point}

\begin{abstract}  
We show that the set of conjugacy classes of cubic polynomials with a prefixed critical point, of preperiod $k\geq 1$, is an irreducible algebraic curve. We also establish an analogous result for quadratic rational maps. We then study a closely related question concerning the irreducibility (over $\Q$) of the set of conjugacy classes of unicritical polynomials, of degree $D\geq 2$, with a preperiodic critical point. Our proofs are purely algebraic. 
\end{abstract} 
%We prove that for $n\geq 1$, the affine conjugacy classes of cubic polynomials with a prefixed critical point of preperiod $n$ form an irreducible algebraic curve. We prove that an analogous result holds for quadratic rational maps. Our proof is completely algebraic. We also study a closely related result: the irreducibility over $\Q$ of conjugacy classes of unicritical polynomials having a preperiodic critical point. 
\maketitle
\tableofcontents

\section*{Introduction}

Let $f:\C\P^1\to \C\P^1$ be a rational map.  A point $z\in \C\P^1$ is 
 \begin{itemize}
\item periodic for $f$ with period $n\geq 1$ if $f^{\circ n}(z) = z$ and $n$ is the least such integer; 
\item preperiodic for $f$ with preperiod $k\geq 0$ if $f^{\circ k}(z)$ is periodic for $f$ and $k$ is the least such integer.
\end{itemize}

The moduli space ${\mathcal P}_3$ of affine conjugacy classes of cubic polynomials is isomorphic to $\C^2$. Similarly, the moduli space ${\mathcal M}_2$ of Möbius conjugacy classes of quadratic rational maps is isomorphic to $\C^2$. In both cases, requiring that one critical point is preperiodic to a cycle of period $n\geq 1$ with preperiod $k\geq 0$ (with $k\neq 1$ in the case of quadratic rational maps) defines an algebraic curve. In \cite{milnorquad} and \cite{milnorcubic}, John Milnor introduced these curves and raised various questions about their geometry. In this article, we prove that the curves consisting of those maps with a prefixed critical point are irreducible. 

We shall first study the case of cubic polynomials. Given $k\geq 0$ and $n\geq 1$, the affine conjugacy classes of cubic polynomials with a critical point preperiodic to a cycle of period $n$ with preperiod $k$ form an algebraic curve ${\mathcal S}_{k,n}\subset {\mathcal P}_3$. The following conjecture goes back to John Milnor \cite[Question 5.3]{milnorcubic} in the case $k=0$. 

 \begin{conjecture}
For $k\geq 0$ and $n\geq 1$, the curve ${\mathcal S}_{k,n}$ is irreducible. 
 \end{conjecture}
 
A proof in the case $k=0$ has recently been announced by Matthieu Arfeux and Jan Kiwi \cite{arfeuxkiwi}; it relies on a result of Mary Rees in \cite{rees}, that the set of fixed points of an endomorphism on a certain Teichmüller space is connected. We shall prove the following result. 
 
 \begin{theorem}\label{theo:maincubic}
 For $k\geq 0$, the curve ${\mathcal S}_{k,1}$ is irreducible. 
 \end{theorem}

Our proof is purely algebraic. It is largely inspired by the proof of Thierry Bousch \cite{bousch} that for $n\geq 1$, the set of $(c,z)\in \C^2$ such that $z$ is periodic of period $n$ for $f_c:w\mapsto  w^2+c$ is irreducible. The proof will be given in \S \ref{sec:cubic}. 

In \S\ref{sec:quadrat}, we explain how the proof presented for cubic polynomials adapts to the case of quadratic rational maps. 
Given $k\geq 0$ with $k\neq 1$ and $n\geq 1$, the Möbius conjugacy classes of quadratic rational maps with a critical point preperiodic to a cycle of period $n$, with preperiod $k$, form an algebraic curve ${\mathcal V}_{k,n} \subset {\mathcal M}_2$.

 \begin{conjecture}
 For $n\geq 1$, the curve ${\mathcal V}_{0,n}$ is irreducible. For $k\geq 2$ and $n\geq 1$, the curve ${\mathcal V}_{k,n}$ is irreducible. 
 \end{conjecture}
 
 In this article, we shall prove the following result. 
 
 \begin{theorem}\label{theo:mainquad}
 For $k\geq 2$, the curve ${\mathcal V}_{k,1}$ is irreducible. 
 \end{theorem}

The proofs of Theorems \ref{theo:maincubic} and \ref{theo:mainquad} rely on the following result due to Vefa Goksel \cite{goksel}. 
Assume $D\in \{2,3\}$. Let $b_1\in \C$ and $b_2\in \C$ be two algebraic numbers such that $0$ is preperiodic to a fixed point of $z\mapsto z^D+b_1$ and $z\mapsto z^D+b_2$, with the same preperiod $k\geq 2$. Then, $b_1$ and $b_2$ are Galois conjugate. 

More generally, if $D\geq 2$ is an integer,  the unicritical polynomials $z\mapsto z^D+b_1$ and $z\mapsto z^D+b_2$ are affine conjugate if and only if $b_1^{D-1} = b_2^{D-1}$. 
John Milnor \cite{milnorunicritical} asked whether one can classify the Galois conjugacy classes of parameters $b^{D-1}$ such that the critical point of $z\mapsto z^D+b$ is preperiodic. 
In \S\ref{sec:unicritical}, we characterize those Galois conjugacy classes when the period is $1$ or $2$ for any prime power $D=p^e$, and when the period is $3$ for $D=2$ and $D=8$. 

\section{Cubic polynomials\label{sec:cubic}}

Every cubic polynomial is affine conjugate to a polynomial of the form 
\[F_{a,b}(z) = z^3-3a^2z+2a^3+b,\quad(a,b)\in \C^2.\]
Those polynomials have critical points at $±a$ and $b= F_{a,b}(a)$ is a critical value. A conjugacy between two such polynomials either preserves or exchanges the two critical points. Consequently, the moduli space ${\mathcal P}_3$ is obtained by identifying $(a,b)$ with $(-a,-b)$. It follows that in order to prove Theorem \ref{theo:maincubic}, it is enough to show that the set $\Sk$ of parameters $(a,b)\in \C^2$ such that $a$ is preperiodic to a fixed point with preperiod $k\geq 0$ is irreducible. 
%{\color{red}Why? Should we talk about the map $(a,b)\mapsto \text{affine conjugacy class of }F_{a,b}\text{ in }\mathcal P_3$  ?}

Note that for $k=0$, the critical point $a$ is fixed if and only if $(a,b)$ belongs to the line $\L_0:=\{b=a\}\subset \C^2$. Thus, ${\cal S}_0=\L_0$ is irreducible.

Note that for $k=1$, the critical value $b=F_{a,b}(a)$ is fixed if and only if 
\[b = F_{a,b}(b) = b^3-3a^2b+2a^3 = b+ (a-b)^2(2a+b).\]
Consequently, ${\cal S}_1 =  \L_1\ssm\L_0 =  \L_1\ssm \bigl\{(0,0)\bigr\}$, with  $\L_1 := \{b=-2a\}\subset \C^2$. Thus, ${\cal S}_1$ is irreducible.

For the remainder of \S\ref{sec:cubic}, we assume that $k\geq 2$. 

\subsection{An equation for $\Sk$\label{sec:equationcubic}}

On the one hand, if $a$ is preperiodic to a fixed point of $F_{a,b}$ with preperiod $k$, then the points $F_{a,b}^{\circ (k-1)}(a)$ and $F_{a,b}^{\circ k}(a)$ are distinct and have the same image under $F_{a,b}$. For $j\geq 0$, let $P_j\in \Z[a,b]$ be the polynomial defined by 
\[P_j(a,b) := F_{a,b}^{\circ j}(a).\]
Then, 
\[P_0(a,b)= a,\quad P_1(a,b) =b,\and P_{j+1} = P_j^3-3a^2P_j+2a^3+b,\]
so that for $j\geq 1$, the polynomial $P_j$ has degree $3^{j-1}$. 
Note that 
\[F_{a,b}(z) - F_{a,b}(w) = (z-w) H(z,w)\with H(z,w) = z^2+zw+w^2-3a^2.\]
Thus, the polynomial
\[Q_k := H(P_{k-1},P_k)\in \Z[a,b]\]
has degree $2\cdot 3^{k-1}$ and vanishes on $\Sk$. 

On the other hand, $H(z,z) =0$ if and only if $z^2=a^2$, i.e. $z=±a$. In particular, if $a=F_{a,b}(a)$, i.e. if $a=b$, then $P_{k-1}(a,b) = P_k(a,b)=a$ and $Q_k(a,b) = 0$. Thus, $b-a$ divides $Q_k$ and so, 
\[Q_k = (b-a) R_k\with R_k\in \Z[a,b].\]
The polynomial $R_k$ has degree $2\cdot 3^{k-1}-1$ and vanishes on $\Sk$. Set 
\[\Sik:= \bigl\{(a,b)\in \C^2~;~R_k(a,b)=0\bigr\}.\]

Then, $\Sk\subset \Sik$. Note that there are points in $\Sik\ssm \Sk$: 
\begin{enumerate}[(i)]
\item\label{case1} either $F_{a,b}^{\circ (k-1)}(a) = F_{a,b}^{\circ k}(a) = a$ in which case $a$ is fixed; 
\item\label{case2} or $F_{a,b}^{\circ (k-1)}(a) = F_{a,b}^{\circ k}(a) = -a$ in which case $-a$ is fixed and $a$ is prefixed to $-a$ with preperiod $j$ for some $j \in \ob2,k-1\cb$.
\end{enumerate}

\begin{remark}
Note that case \eqref{case1} occurs if and only if $a=b=0$, i.e., $\Sik\cap {\cal S}_1 = \bigl\{(0,0)\bigr\}$.
Indeed, for $j\geq 1$, 
\[\frac{\partial P_{j+1}}{\partial b} = 3(P_j^2-a^2)\frac{\partial P_j}{\partial b}+1.\]
Since $P_j(a,a) = a$, it follows by induction that 
\[\frac{\partial P_j}{\partial b}(a,a) = 1.\]
Since
\[\frac{\partial Q_k}{\partial b} = (2P_{k-1}+P_k)\frac{\partial P_{k-1}}{\partial b} + (P_{k-1}+2P_k) \frac{\partial P_k}{\partial b},\]
we deduce that 
\[R_k(a,a) = \frac{\partial Q_k}{\partial b}(a,a) = 6a.\]
Thus, on the line $\{a=b\}\subset \C^2$, the polynomial $R_k$ only vanishes at $(0,0)$. 
\end{remark}

\begin{figure}[h] %  figure placement: here, top, bottom, or page
   \centering
   \includegraphics[width=2.5in]{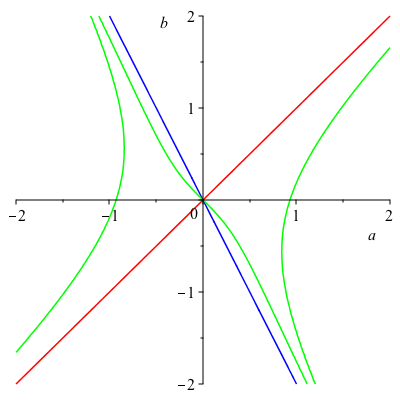} 
   \caption{Three curves drawn in $\mathbb R^2$: ${\cal S}_0$ is red, ${\cal S}_1$ is blue, and ${\cal S}_2$ is green.}
   \label{cubicpic:color}
\end{figure}

%\begin{figure}[htbp] %  figure placement: here, top, bottom, or page
%   \centering
%   \includegraphics[width=2.5in]{cubic_color.jpg} 
%   \caption{Three curves drawn in $\mathbb R^2$: ${\cal S}_0$, ${\cal S}_1$, and ${\cal S}_2$.}
%   \label{cubicpic:grey}
%\end{figure}
%

Theorem \ref{theo:maincubic} is a corollary of the following result, the proof of which  occupies the remainder of \S\ref{sec:cubic}. 

\begin{proposition}\label{prop:irredC}
For $k\geq 2$, the polynomial $R_k\in \Z[a,b]$ is irreducible over $\C$. 
\end{proposition}

\subsection{Behavior near the origin}%Reduction to irreducibility over $\Q$
\label{sec:irredQ}

We now show that in order to prove Proposition \ref{prop:irredC}, it is enough to prove that $R_k$ is irreducible over $\Q$. 

\begin{proposition}
The polynomial $R_k\in \Z[a,b]$ is irreducible over $\C$ if and only if it is irreducible over $\Q$. 
\end{proposition}

\begin{proof}
To prove the proposition, we shall use the following criterion. 

\begin{lemma}\label{lem:criterion}
Let $R\in \Q[a,b]$ be a polynomial vanishing at the origin with nonzero linear part. 
Then, $R$ is irreducible over $\C$ if and only if $R$ is irreducible over $\Q$. 
\end{lemma}

\begin{proof}
Clearly, if $R$ is irreducible over $\C$, then it is irreducible over $\Q$. 

Conversely, suppose that $R$ is irreducible over $\Q$. We will show that $R$ is irreducible over $\C$. Suppose that $R=S\cdot T$ where $S\in \C[a,b]$ is irreducible and vanishes at the origin. Such a polynomial $S$ exists because $R$ vanishes at the origin. It then follows that $T\in \C[a,b]$ does not vanish at the origin, since otherwise, the linear part of $R$ at the origin would vanish. Multiplying $S$ by a nonzero constant, we may assume that $T(0,0)=1$. In that case, the linear parts of  $R$ and $S$ at the origin coincide. 

Since $R\in \Q[a,b]$, the polynomials $S$ and $T$ have algebraic coefficients. We claim that the coefficients of $S$ are in fact rational. Indeed, assume $\sigma\in \Gal$. Let $S^\sigma$ be the image of $S$ under the action of $\sigma$. Then, $S^\sigma$ is an irreducible factor of $R^\sigma$ and $R^\sigma=R$ since $R\in \Q[a,b]$. Note that $S$ and $S^\sigma$ are equal up to multiplication by a constant since otherwise, $S\cdot S^\sigma$ would divide $R$, and the linear part of $R$ at the origin would vanish.  In addition, the linear part of $S^\sigma$ is equal to the linear part of $R^\sigma = R$. Thus, $S^\sigma = S$. Since this holds for all $\sigma \in \Gal$, the coefficients of $S$ are rational. 

Since $S\in \Q[a,b]$ is a factor of $R$ and since $R$ is irreducible over $\Q$, we have that $S=R$. This completes the proof since $S$ is irreducible over $\C$ by assumption. 
\end{proof}

To apply this lemma, we need to study the behavior of $R_k$ at the origin. 

\begin{lemma}\label{lem:near0}
The homogeneous part of least degree of $R_k$ is $3(a+b)$. 
\end{lemma}

\begin{proof}
An elementary induction on $j\geq 1$ shows that the homogeneous part of least degree of $P_j$ is $b$. As a consequence, the homogeneous part of least degree of $Q_k$ is $3b^2-3a^2$. Factoring out $b-a$ to get $R_k$ yields the required result. 
\end{proof}

Thus, $R_k$ vanishes at the origin with nonzero linear part $(a,b)\mapsto 3(a+b)$. 
This completes the proof of the proposition. 
\end{proof}

\subsection{The family $z^3+b$, $b\in \C$\label{sec:zcubeplusc}}

We now study the intersection of $\Sk$ with the line 
$\L_2:=\{a=0\}\subset \C^2$. Note that the map $f_b:=F_{0,b}$ is a unicritical polynomial: 
\[f_b(z) = z^3+b.\]
For $j\geq 1$, define $p_j\in \Z[b]$ by 
\[p_j(b) := P_j(0,b)\quad\text{so that}\quad p_1 = b\and p_{j+1}= p_j^3+b.\] 
Let $q_k\in \Z[b]$ and $r_k\in \Z[b]$ be defined by 
\[q_k(b):=Q_k(0,b)\and r_k(b):=R_k(0,b),\]
so that 
\[q_k = p_{k-1}^2+p_{k-1}p_k+p_k^2 \and q_k = br_k.\]
An easy induction on $j\geq 1$ shows that $p_j$ is a monic polynomial of degree $3^{j-1}$ with least degree term $b$. 
It follows that $q_k$ is a monic polynomial of degree $2\cdot 3^{k-1}$ with least degree term $3b^2$. Thus, $q_k=br_k=b^2 s_k$ where $s_k\in \Z[b]$ is a monic polynomial of degree $2\cdot 3^{k-1}-2$ with $s_k(0)=3$.  
The proof of the following result goes back to \cite{goksel} (see also \S\ref{sec:prefixedD}). 

\begin{proposition}
For $k\geq 2$,  the polynomial $s_k\in \Z[b]$ is irreducible over $\Q$.
\end{proposition}

\begin{proof}
Working in ${\mathbb F}_3[b]$, we have that $(x+y)^3 \equiv x^3+y^3\mod 3$. An elementary induction on $j\geq 1$ yields 
\[p_j \equiv b^{3^{j-1}} +b^{3^{j-2}}+  \cdots + b^3+  b \mod{3}. \]
% \cdots +b^{3^{n-2}}+ b^{3^{n-1}}\mod{3}. \]
It follows that 
\[p_{k}-p_{k-1} \equiv b^{3^{k-1}}\mod{3}\and (p_k-p_{k-1}) q_k  = (p_k-p_{k-1})^3 \equiv b^{3^k}\mod 3.\]
Thus, 
\[q_k \equiv b^{2\cdot 3^{k-1}}\mod 3\and s_k \equiv b^{2\cdot 3^{k-1}-2}\mod 3.\]
Since $s_k(0) = 3$ is not a multiple of $9$, the Eisenstein criterion implies that $s_k$ is irreducible over $\Q$. 
\end{proof}

\subsection{Behavior near infinity}

We now study the behavior of $R_k$ when $a$ or $b$ is large. 

\begin{lemma}\label{lem:intersectinfinity}
The homogeneous part of greatest degree of $R_k$ is 
\[(b-a)^{4\cdot 3^{k-2}-1}\cdot (2a+b)^{2\cdot 3^{k-2}}.\]
\end{lemma}

\begin{proof}
We first determine the homogeneous part $H_k$ of greatest degree of $P_j$ for $j\geq 2$. 
Since \[P_2 = b^3-3a^2b+2a^3+b = (b-a)^2(2a+b)+b\and P_{j+1} = P_j^3-3a^3P_j + 2a^3+b,\]
we have $H_2 = (b-a)^2(2a+b)$ and an elementary induction on $j\geq 2$ yields that $H_j = (H_2)^{3^{j-2}}$.
It follows that the homogeneous part of greatest degree of $Q_k= P_{k-1}^2 + P_{k-1}P_k+P_k^2 -3a^2$ is 
$ (H_2)^{2\cdot 3^{k-2}} = (b-a)^{4\cdot 3^{k-2}}\cdot (2a+b)^{2\cdot 3^{k-2}}$. Factoring out $b-a$ to get $R_k$ yields the required result.
\end{proof}

Let us embed $\C^2$ in $\C\P^2$ in the usual way, sending $(a,b)$ to $[a:b:1]$.

\begin{corollary}\label{coro:intersectinfinity}
The closure of $\Sik$ in $\C\P^2$ intersects the line at infinity at only two points: $[1:1:0]$ with multiplicity $4\cdot 3^{k-2}-1$, and $[1:-2:0]$ with multiplicity $2\cdot 3^{k-2}$. 
\end{corollary}

\subsection{Irreducibility over $\Q$\label{sec:irreducibleQ}}

We may now complete the proof of Proposition \ref{prop:irredC}. 

\begin{proposition}
For $k\geq 2$, the polynomial $R_k\in \Z[a,b]$ is irreducible over $\Q$. 
\end{proposition}

\begin{proof}
Assume by contradiction that $R_k = T_1\cdot T_2$ with $T_1\in \Z[a,b]$, $T_2\in \Z[a,b]$, ${\rm degree}(T_1)<{\rm degree}(R_k)$, and ${\rm degree}(T_2)<{\rm degree}(R_k)$. 

We first prove that either $T_1$ or $T_2$ must have degree $1$. 
Let $t_1\in \Z[b]$ and $t_2\in \Z[b]$ be defined by  
\[t_1(b) := T_1(0,b)\and t_2(b):= T_2(0,b).\]
Then, $r_k = t_1\cdot t_2$ with ${\rm degree}(t_1)\leq {\rm degree}(T_1)<{\rm degree}(R_k) = {\rm degree}(r_k)$. Similarly, ${\rm degree}(t_2)< {\rm degree}(r_k)$. Since $r_k = bs_k$ with $r_k$ monic and $s_k$  irreducible over $\Q$, exchanging $T_1$ and $T_2$ if necessary, this implies that $t_1=±b$ and $t_2=±s_k$. Then, ${\rm degree}(T_2)\geq {\rm degree}(s_k) = {\rm degree}(R_k)-1$ and ${\rm degree}(T_1)=1$. 

According to Lemma \ref{lem:near0}, the homogeneous part of least degree of $R_k$ is $3(a+b)$. Thus, $T_1$ divides $3(a+b)$; in fact, since $t_1=±b$, we have that $T_1=±(a+b)$.  So, the closure of $\Sik$ in $\C\P^2$ intersects the line at infinity at the point $[1:-1:0]$. This contradicts Corollary \ref{coro:intersectinfinity}. 
\end{proof}

\section{Quadratic rational maps\label{sec:quadrat}}

To prove Theorem \ref{theo:mainquad}, it is convenient to work in a space of dynamically marked quadratic rational maps. 
A quadratic rational map whose conjugacy class belongs to ${\mathcal V}_{k,1}$ with $k\geq 2$ has a critical point $\omega$ whose orbit contains a fixed point $\alpha$. There is a fixed point $\beta\neq \alpha$ since otherwise, $\alpha$ would be a triple fixed point and its parabolic basin would contain both critical orbits. Note that $\beta\neq \omega$ since $\omega$ is not fixed. The conjugacy class may therefore be represented by a rational map $f$ such that 
\[\alpha=0,\quad \beta=\infty\and  \omega = 1.\]
The critical value $a=f(1)$ belongs to $\C\ssm\{0\}$ and $f^{-1}(0) = \{0,b\}$ with $b\in \C\ssm \{1\}$. 
So, the rational map is 
\[G_{a,b}(z) := \frac{az(b-z)}{1+(b-2)z}\with (a,b)\in \Lambda:=\bigl(\C\ssm\{0\}\bigr)\times \bigl(\C\ssm\{1\}\bigr).\]
In addition, $(a,b)$ belongs to the curve
\[\Ck:=\bigl\{(a,b)\in  \Lambda~;~G_{a,b}^{\circ (k-2)}(a) = b\bigr\}.\]
Conversely, if $(a,b)$ belongs to the curve $\Ck$, then the conjugacy class of $G_{a,b}$ belongs to ${\mathcal V}_{k,1}$. So, in order to prove Theorem \ref{theo:mainquad}, it is enough to prove that the curve $\Ck$ is irreducible. 

\begin{remark}
A generic conjugacy class in ${\mathcal V}_{k,1}$ has two representatives in $\Ck$ corresponding to the choice of the marked fixed point $\beta$. It follows that the quotient map $\Ck\to {\mathcal V}_{k,1}$ has degree $2$. 
\end{remark}

\subsection{An equation for $\Ck$\label{sec:equation}}

Here, we shall define a polynomial $R_k\in \Z[a,b]$ vanishing on $\Ck$. This polynomial shall not be confused with the polynomial $R_k$ defined in \S\ref{sec:cubic}. However, since they play parallel roles, we keep the same notation. 
Let us first observe that for $j\geq 2$, 
\[G_{a,b}^{\circ (j-2)}(a) = \frac{P_j(a,b)}{Q_j(a,b)}\]
where $P_j\in \Z[a,b]$ and $Q_j\in \Z[a,b]$ are defined recursively by 
\[P_2 = a,\quad Q_2 = 1,\quad P_{j+1} = aP_j\cdot (bQ_j-P_j) \and Q_{j+1} = Q_j^2+(b-2)P_jQ_j.\]
So, $\Ck$ is the set of parameters $(a,b)\in \Lambda$ such that 
\[R_k(a,b)=0\with R_k := P_k-bQ_k\in \Z[a,b].\]
This shows that $\Ck$ is an algebraic subset of $\Lambda$ and that Theorem \ref{theo:mainquad} follows from the following 
result.
 
\begin{proposition}\label{prop:irreducquad}
For $k\geq 2$, the polynomial $R_k\in \Z[a,b]$ is irreducible over $\C$. 
\end{proposition}

Note that $R_2 = a-b$ is irreducible over $\C$. For the remainder of \S\ref{sec:quadrat}, devoted to the proof of Proposition \ref{prop:irreducquad}, we assume that $k\geq 3$. 
%We shall use the following observation. 

\begin{figure}[h] %  figure placement: here, top, bottom, or page
   \centering
   \includegraphics[width=2.5in]{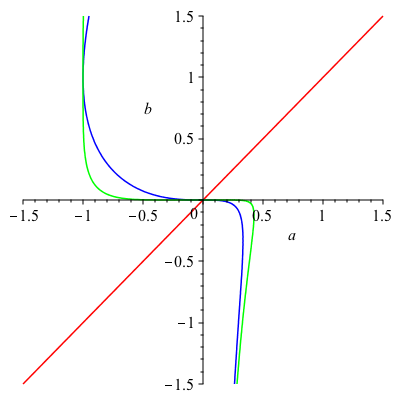} 
   \caption{Three curves drawn in $\mathbb R^2$: ${\cal V}_2$ is red, ${\cal V}_3$ is blue, and ${\cal V}_4$ is green.}
   \label{quadpic:color}
\end{figure}

%\begin{figure}[htbp] %  figure placement: here, top, bottom, or page
%   \centering
%   \includegraphics[width=2.5in]{quad_grey.jpg} 
%   \caption{Three curves drawn in $\mathbb R^2$: ${\cal V}_2$,  ${\cal V}_3$, and ${\cal V}_4$.}
%   \label{quadpic:grey}
%\end{figure}
%

\subsection{Behavior near the origin}

As in \S\ref{sec:irredQ}, we first prove that it is enough to show that $R_k$ is irreducible over $\Q$. And here also, we deduce this from Lemma \ref{lem:criterion}, studying the behavior of $R_k$ near the origin. There is however a fundamental difference between the two approaches, even if this does not appear in the proof. In the case of cubic polynomials, the origin corresponds to the cubic polynomial $z\mapsto z^3$ which belongs to the family we are studying, whereas here, the origin does not belong to our parameter space $\Lambda$. 

\begin{lemma}
For $k\geq 3$, the homogeneous part of least degree of $R_k$ is $-b$. 
\end{lemma}

\begin{proof}
An elementary induction shows that for $j\geq 2$, the homogeneous part of least degree of $P_j$ is $a^{j-1}b^{j-2}$ and the homogeneous part of least degree of $Q_j$ is $1$. The result follows immediately. 
\end{proof}

 As a consequence $R_k\in \Z[a,b]$ vanishes at the origin with nonzero linear part. According to Lemma \ref{lem:criterion}, the polynomial $R_k$ is irreducible over $\C$ if and only if it is irreducible over $\Q$.
 
\subsection{The family $az(2-z)$, $a\in \C$.\label{eq:irreducibleQ}}

We now study the intersection of $\Ck$ with the line $\L:=\{b=2\}\subset \C^2$. Note that the map $g_a:=G_{a,2}$ is a quadratic polynomial: 
\[g_a(z) = az(2-z).\]
For $j\geq 2$, define $p_j\in \Z[a]$ and $q_j\in \Z[a]$  by 
\[p_j(a) := P_j(a,2),\quad q_j(a):=Q_j(a,2)\]
so that 
\[ p_2 = a,\quad p_{j+1} = -ap_j^2+2ap_j,\quad q_1 = 1\and q_{j+1} = q_j^2.\]
In particular, for $j\geq 3$, the polynomial $-p_j$ is monic with degree $2^{j-1}-1$, and its constant coefficient is $0$; and $q_j=1$. 
Let $r_k\in \Z[a]$ be defined by 
\[r_k(a):= R_k(a,2)\quad \text{so that }\quad r_k = p_k-2.\]
Then, $r_k$ has degree $2^{k-1}-1$ and its constant coefficient is $-2$. 

\begin{lemma}
The degree of $R_k$ is $2^{k-1}-1$. 
\end{lemma}

\begin{proof}
An elementary induction shows that the degree of $P_k$ is at most $2^{k-1}-1$ and the degree of $Q_k$ is at most $2^{k-1}-2$. Consequently, the degree of $R_k$ is at most  $2^{k-1}-1$. 

Since the polynomial $p_k$ has degree $2^{k-1}-1$, the polynomial $r_k = p_k-2$ also has degree $2^{k-1}-1$. Thus, 
\[2^{k-1}-1 = {\rm degree}(r_k) \leq  {\rm degree}(R_k) \leq 2^{k-1} - 1\]
 and the result follows.
\end{proof}

The proof of the following result goes back to \cite{goksel} (see also \S\ref{sec:prefixedD}). 

\begin{proposition}
For all $k\geq 2$, the polynomial $r_k\in \Z[a]$ is irreducible over $\Q$. 
\end{proposition}

\begin{proof}
Working in ${\mathbb F}_2[a]$, we have that for $j\geq 2$, 
\[p_{j+1} \equiv a p_j^2\mod{2}\quad\text{so that}\quad r_k \equiv p_k \equiv a^{2^{k-1}}\mod{2}.\]
The constant coefficient of  $r_k$ is $-2$. 
It follows from the Eisenstein criterion that $r_k$ is irreducible over $\Q$. 
\end{proof}
%
%\begin{corollary}
%The intersection $\Gamma_n\cap \L$ is irreducible over $\Q$. 
%\end{corollary}
%
%Embed $\Lambda$ in $\C\P^2$ in the usual way, sending $(a,b)$ to $[a:b:1]$, and let 
%$\overline\Sigma_n$ and $\overline \L$ be the respective closures of $\Gamma_n$ and $\L$ in $\C\P^2$.
%
%\begin{proposition}
%The intersection $\overline\Sigma_n\cap \overline \L$ is contained in $\L$.
%\end{proposition}
%
%\begin{proof}
%Note that $\overline \L\ssm \L$ consists of two points : the points $[0:2:1]$ and $[1:0:0]$. 
%
%On the one hand, if $(a,b)\in \Lambda$ is close to $[0:2:1]$, i.e., $a$ is close to $0$ and $b$ is close to $2$, then $G_{a,b}^{\circ n}(a)$ is close to $g_{0,2}^{\circ n}(a) = 0\neq 2$. This shows that $[0:2:1]$ is not contained in $\overline\Sigma_n$. 
%
%On the other hand, if $(a,b)\in \Lambda$ is sufficiently close to $[1:0:0]$, then $|a|>10$, $|b/a|<1/10$ and $|b-2|/|a|<1/10$. Then, if $|z|\geq |a|$, we have that 
%\[\bigl|G_{a,b}(z)\bigr| \geq |z| \frac{1-|b/a|}{1/|a|^2+|b-2|/|a|}\geq \frac{1-1/10}{1/100+1/10}|z|= \frac{90}{11}|z|.\]
%Thus, the orbit of $a$ is attracted by $\infty$. So, the critical point $1$ is not preperiodic to the fixed point $0$ 
%and $(a,b)$ is not contained in $\Gamma_n$. This shows that $[1:0:0]$ is not contained in $\overline\Sigma_n$.
%\end{proof}

\subsection{Irreducibility over $\Q$}

We may now complete the proof of Proposition \ref{prop:irreducquad}.

\begin{proposition}
The polynomial $R_k\in \Z[a,b]$ is irreducible over $\Q$. 
\end{proposition}

\begin{proof}
Assume by contradiction that $R_k = T_1\cdot T_2$ with $T_1\in \Z[a,b]$, $T_2\in \Z[a,b]$, ${\rm degree}(T_1)<{\rm degree}(R_k)$ and ${\rm degree}(T_2)<{\rm degree}(R_k)$. Consider the polynomials $t_1\in \Z[a]$ and $t_2\in \Z[a]$ defined by  
\[t_1(a) := T_1(a,2)\and t_2(a):=T_2(a,2).\]
Then, $r_k = t_1\cdot t_2$ with ${\rm degree}(t_1)\leq {\rm degree}(T_1)<{\rm degree}(R_k) = {\rm degree}(r_k)$. Similarly, ${\rm degree}(t_2)< {\rm degree}(r_k)$. This is not possible since $r_k$ is  irreducible over $\Q$. 
\end{proof}

\section{Unicritical polynomials\label{sec:unicritical}}

The previous discussion motivates a more systematic study of irreducibility over $\Q$ within families of unicritical polynomials. This section is devoted to such a study. It can be read independently of the rest of the article. 
Consider the polynomials $f_a:\C\to \C$ defined by 
\[f_a(z) = az^D+1,\quad a\in \C.\]
The polynomial $f_a$ is unicritical: it has a unique critical point at $z=0$. We are interested in parameters $a$ such that the critical point is preperiodic for $f_a$. Note that the preperiod $k$ cannot be equal to $1$.

For $n\geq 1$, let $P_n\in \Z[a]$ be the polynomial 
\[P_n(a):= f_a^{\circ n}(0).\]
Gleason observed that the discriminant of $P_n$ is $1\mod D$, and thus $P_n$ has simple roots. It follows that
\[P_n = \prod_{m|n} R_m\quad \text{with}\quad R_n := \prod_{m | n} P_m^{\mu(n/m)}\in \Z[a],\]
where $\mu$ is the Möbius function defined by $\mu(i) =(-1)^j$ if $i$ is the product of $j$ distinct primes with $j\geq 0$ and $\mu(i)=0$ otherwise. 
For example, 
\[R_1=P_1 = 1,\quad R_2 = P_2 = a+1\quad \text{and}\quad R_3 = P_3 = a(a+1)^D + 1.\]
It is conjectured that when $D=2$, the polynomials $R_n$ are irreducible over $\Q$ for all $k\geq 2$.
The following result shows that this is not true when  $D\equiv 1 \mod 6$.  

\begin{proposition}[\cite{buff}]
The polynomial $R_3$ is irreducible over $\Q$ if and only if $D$ is not congruent to $1$ modulo 6. When $D\equiv 1 \mod 6$, the polynomial $R_3$ has exactly two irreducible factors over $\Q$, one of which is $a^2+a+1$. 
\end{proposition}

Assume now that $0$ is preperiodic for $f_a$ with preperiod $k\geq 2$ and period $n\geq 1$. Then, 
\begin{equation}\label{eq:omega}
f_a^{\circ (k+n-1)}(0) = \omega f_a^{\circ (k-1)}(0)\quad \text{with}\quad \omega^D=1\quad \text{and}\quad \omega\neq 1.
\end{equation}
In fact, Equation \eqref{eq:omega} is satisfied if and only if either $0$ is periodic for 
$f_a$ with period dividing $\gcd(n,k-1)$, or $0$ is preperiodic for $f_a$ with preperiod $k$ and period dividing $n$. 

For $k\geq 2$, $n\geq 1$ and $d\geq 2$ dividing $D$, we therefore consider the monic polynomial $R_{k,n,d}$ whose roots are the parameters $a\in \C$ such that 
\begin{itemize}
\item $0$ is preperiodic for $f_a$ with preperiod $k$ and period $n$,  and
\item Equation \eqref{eq:omega} is satisfied for some primitive $d$-th root of unity $\omega$.
\end{itemize} 
We claim that $R_{k,n,d}\in \Z[a]$. Indeed, let $\Phi_d\in \Z[X,Y]$ be the (homogenized) $d$-th cyclotomic polynomial: if $\Omega_d$ is the set of primitive $d$-th roots of unity, then 
\[\Phi_d := \prod_{\omega\in \Omega_d} (X-\omega Y).\]
Let $P_{k,n,d}\in \Z[a]$ be the polynomial defined by 
\[P_{k,n,d} :=\Phi_d(P_{k+n-1},P_{k-1}) = \prod_{\omega\in \Omega_d} (P_{k+n-1}-\omega P_{k-1}).\]
The polynomial $P_{k+n-1}-\omega P_{k-1}$ has simple roots (see \cite{buff} for example). In addition, the common roots of $P_{k+n-1}$ and $P_{k-1}$ are the roots of $P_{\gcd(n,k-1)}$.
It follows that the multiple roots of $P_{k,n,d}$ are the roots of $P_{\gcd(n,k-1)}$ with multiplicities $\varphi(d) =\deg(\Phi_d)$, where 
$\varphi$ is the Euler totient function. As a consequence,  
\begin{equation}\label{eq:fkpd}
P_{k,n,d} = P_{{\rm gcd}(n,k-1)}^{\varphi(d)} \cdot \prod_{m|n} R_{k,m,d}
\end{equation}
 and according to the Möbius Inversion Formula,
 \[ R_{k,n,d}=\prod_{m|n} \left(\frac{P_{k,m,d}}{P_{{\rm gcd}(m,k-1)}^{\varphi(d)}}\right)^{\mu(n/m)}\in \Z[a]. 
\]
We shall also consider the polynomials $P_{k,n}\in \Z[a]$ and $R_{k,n}\in \Z[a]$ defined by 
\[P_{k,n} := \prod_{d|D\atop d\neq 1} P_{k,n,d} = \frac{P_{k+n-1}^D - P_{k-1}^D}{P_{k+n-1}-P_{k-1}} = \sum_{i+j=D-1} P_{k+n-1}^i\cdot  P_{k-1}^j.\]
and 
\begin{equation}\label{eq:buff}
R_{k,n} := \prod_{d|D\atop d\neq 1} R_{k,n,d} \in \Z[a] \quad\text{so that} \quad P_{k,n} =  P_{{\rm gcd}(n,k-1)}^{D-1}\cdot \prod_{m|n} R_{k,m}.\end{equation}

We shall study the following conjecture of John Milnor \cite{milnorunicritical} (compare with \cite{hutz}). 

\begin{conjecture}
For all $k\geq 2$, $n\geq 1$, and $d\geq 2$ that divide $D\geq 2$, the polynomial $R_{k,n,d}$ is irreducible over $\Q$.
\end{conjecture}

There are few cases where the expression of $R_{k,n,d}$ is sufficiently simple so that existing results in the literature directly apply (see Appendix \ref{app:examples}). 
In this article, our study is inspired by the following fundamental result of Vefa Goksel. 

\begin{theorem}[\cite{goksel}]
If $D$ is a prime number, then $R_{k,1}(c^{D-1})\in \Z[c]$ is irreducible for all $k\geq 2$. If $D=2$, then $R_{k,2}$ is irreducible for all $k\geq 2$. 
\end{theorem}

Our main result is the following. In the whole article, $p$ is a prime number. 

\begin{theorem}\label{theo}
Assume $D=p^e$ is a prime power. Then $R_{k,1,d}$ is irreducible over $\Q$ for all $k\geq 2$, and for all $d\geq 2$ that divide $D$. More generally, if $n\geq 2$ and the polynomial $R_n\mod p$ is irreducible over $\F_p$, then $R_{k,n,d}$ is irreducible over $\Q$ for all $k\geq 2$, and for all $d\geq 2$ that divide $D$. 
\end{theorem}

\begin{corollary}
Assume $D=p^e$ is a prime power. Then $R_{k,2,d}$ is irreducible over $\Q$ for all $k\geq 2$, and for all $d\geq 2$ that divide $D$.
\end{corollary}

\begin{proof}
The reduction of $R_2=a+1$ modulo $p$ is irreducible over $\F_p$. 
\end{proof}

\begin{corollary}
 If $D=2$ then $R_{k,3}$ is irreducible over $\Q$ for all $k\geq 2$. 
\end{corollary}

\begin{proof}
If $D=2$, then $R_3=a(a+1)^2+1\equiv 1+a+a^3\mod 2$ and $R_3 \mod 2$ is irreducible over $\F_2$.
\end{proof}

\begin{corollary}
 If $D=8$, then $R_{k,3,2}$, $R_{k,3,4}$ and $R_{k,3,8}$ are irreducible over $\Q$ for all $k\geq 2$. 
\end{corollary}

\begin{proof}
If $D=8$, then $R_3 = a(a+1)^8+1\equiv 1+a+a^9\mod 2$ and $R_3 \mod 2$ is irreducible over $\F_2$.
\end{proof}

\begin{remark}
The only values of $D = p^e$ and $n\geq 2$ for which the polynomial $R_n \mod p$ is irreducible over $\F_p$ are the one listed previously: $n=2$ for any prime power degree $D$, and $n=3$ for both $D=2$ and $D=8$ (see Appendix \ref{app:irreduc}). 
\end{remark}

Our proof of Theorem \ref{theo} relies on the following two results (see \S \ref{sec:general}). 

\begin{lemma}\label{lem:resultgkmd}
Assume $d\geq 2$ divides $D\geq 2$. Assume $k\geq 2$, $n\geq 1$ and $m\geq 1$. Then, 
\[\result(R_{k,m,d},R_n) = \begin{cases}
±p^{\deg(R_n)}&\text{if }n=m\text{ and }d=p^e\text{ is a prime power}\\
±1&\text{otherwise}.\end{cases}\]
\end{lemma}

\begin{lemma}\label{lem:gknd}
Assume $D=p^e$ is a prime power and $d\geq 2$ is a divisor of $D$. Then for all $k\geq 2$, the polynomials $R_{k,1,d} \mod p$ are powers of $a\in \F_p[a]$; and for all $k\geq 2$ and all $n\geq 2$, the polynomials $R_{k,n,d} \mod p$ are powers of $R_n\mod p$.
\end{lemma}

\begin{remark}
Lemma \ref{lem:resultgkmd} shows a connection between  the polynomials $R_{k,n,d}$ and the polynomials $R_n$, valid for all degrees $D\geq 2$. Lemma \ref{lem:gknd} shows a stronger connection between these polynomials, but only valid for prime power degrees $D=p^e$. We think that it is worth investigating what this relation becomes when $D$ is no longer a prime power. 
\end{remark}

\subsection{The critical orbit}

\noindent We shall first study some properties of the polynomials $P_k\in \Z[a]$. 
 Recall that by definition, for all $k\geq 1$, 
\[P_k(a) := f_a^{\circ k}(0).\]
For $k\geq 0$, set 
\[N_k:= \frac{D^k-1}{D-1}\quad \text{so that}\quad 1+D N_k = \frac{D-1 + D^{k+1} - D}{D-1} = N_{k+1}.\]

\begin{lemma}\label{lem:fk}
For all $k\geq 1$, the polynomial $P_k$ has constant coefficient $1$ and is monic of degree $N_{k-1}$.
\end{lemma}

\begin{proof}
First, note that $P_1=1$ and for all $k\geq 1$, $P_{k+1} = a P_k^D+1$. It follows that the constant coefficient of $P_{k+1}$ is $1$. 
Second, let us prove by induction on $k\geq 1$ that $P_k$ is monic of degree $N_{k-1}$. 
The property holds for $k=1$: indeed, $P_1=1$ and $N_0=0$. Now, if the result holds for some integer $k\geq 1$, then $P_{k+1} = a P_k^D+1$ 
is monic of degree $1+D N_{k-1} = N_k$. 
\end{proof}

\begin{lemma}\label{lem:fkfk}
Assume $D=p^e$ is a prime power.  For all $k\geq 1$, 
\[P_{k+1} - P_k \equiv a^{N_k}\mod p.\]
\end{lemma}

\begin{proof}
We prove the result by induction on $k\geq 1$. 
For $k=1$, 
\[P_2-P_1 = a+1-1 = a =a^{N_1}. \]
Now, assume the property holds for some $k\geq 1$. Since $D=p^e$, 
\begin{align*}
P_{k+2} - P_{k+1} &= (aP_{k+1}^D +1) - (aP_k^D+1) \\
&= a\cdot (P_{k+1}^D-P_k^D) \equiv a\cdot (P_{k+1}-P_k)^D\mod p.\end{align*}
Thus, 
\[P_{k+2} - P_{k+1} \equiv a^{1+D N_k} \mod p\equiv a^{N_{k+1}} \mod p.\qedhere\]
\end{proof}

We conclude this section by the following observation due to Poonen. 

\begin{lemma}[Poonen]\label{lem:poonen}
For $m\neq n$, we have that  $\result(R_m,R_n)=\pm 1$. 
\end{lemma}
\begin{proof} Assume $n>m$. 
It is not hard to see by induction on $k\geq 1$, that 
\[P_{m+k}\equiv P_k\mod{P_m^D}.\]
Indeed, $P_{m+1} = aP_m^D + 1 = P_1 + a P_m^D$ and if $P_{m+k} \equiv P_k\mod{P_m^D}$, then 
\[P_{m+k+1} = a P_{m+k}^D + 1 \equiv aP_k^D+1 \mod{P_m^D} \equiv P_k\mod{P_m^D}.\]
This implies that, $P_{m n}\equiv P_m \mod{P_m^D}$. Since $m<n$,  $P_m R_n$ divides $P_{mn}$. So, there are polynomials $A\in \Z[a]$ and $B\in \Z[a]$ such that 
\[ AP_m R_n = P_{mn} = P_m + B P_m^D.\]
Dividing by $P_m$ yields $AR_n-BP_m^{D-1} = 1$. It follows that $R_m$ and $R_n$ are relatively prime in $\Z[a]$ and $\result(R_m,R_n)=\pm 1$. 
\end{proof}

\subsection{When the critical point is preperiodic to a fixed point\label{sec:prefixedD}}

\noindent As a warm up, let us first prove the following proposition that is due to Vefa Goksel. Our proof differs significantly from the one given in \cite{goksel}.

\begin{proposition}
If $D$ is prime, then $R_{k,1}$ is irreducible over $\Q$ for all $k\geq 2$. 
\end{proposition}

\begin{proof}
Our proof relies on the following two lemmas. 

\begin{lemma}\label{lem:fkp}
For $k\geq 2$ and $n\geq 1$, the polynomial $P_{k,n}$ has constant coefficient $D$ and is monic of degree $(D-1)N_{k+n-2}$.
\end{lemma}

\begin{proof}
According to Lemma \ref{lem:fk}, if $i+j=D-1$, the polynomial $P_{k+n-1}^i \cdot P_{k-1}^j$ has constant coefficient $1$ and is monic of degree 
\[i\cdot N_{k+n-2} + j\cdot N_{k-2} \leq  (D-1)N_{k+n-2}\]
with equality if and only if $i = D-1$ and $j=0$. There are $D$ pairs $(i,j)\in \N^2$ such that $i+j=D-1$. Only one pair contributes to the leading term. Thus the polynomial is monic. Every pair contributes to the constant coefficient, which therefore is equal to $D$. 
\end{proof}

\begin{lemma}\label{lem:fk1}
If $D$ is prime, then for all $k\geq 1$, 
\[R_{k,1} = P_{k,1} \equiv a^{(D-1)N_{k-1}}\mod D.\]
\end{lemma}

\begin{proof}
Assume $D$ is prime.  On the one hand, according to  Lemma \ref{lem:fkfk}: 
\begin{equation}\label{eq:fkfk}
P_k^D - P_{k-1}^D \equiv (P_k-P_{k-1})^D \mod D\equiv a^{DN_{k-1}}\mod D.
\end{equation}
On the other hand, by definition of $P_{k,1}$:
\[P_k^D - P_{k-1}^D = (P_k-P_{k-1})\cdot P_{k,1} \equiv a^{N_{k-1}} P_{k,1} \mod D.\]
As a consequence, 
\[a^{N_{k-1}} P_{k,1} \equiv a^{DN_{k-1}}\mod D\quad \text{so that}\quad P_{k,1} \equiv a^{(D-1)N_{k-1}}\mod D.\qedhere\]
\end{proof}

The proposition now follows from the Eisenstein criterion: $R_{k,1}$ is monic, $D$ divides all the coefficients except the one of the leading term, and $D^2$ does not divide the constant coefficient. 
\end{proof}

\subsection{The general case\label{sec:general}}

\noindent This section is devoted to the proof of Theorem \ref{theo}. We first prove Lemmas  \ref{lem:resultgkmd} and \ref{lem:gknd}.

\begin{proof}[Proof of Lemma \ref{lem:resultgkmd}]
Assume $d\geq 2$ divides $D\geq 2$, $k\geq 2$, $n\geq 1$ and $m\geq 1$. We need to show that
\[\result(R_{k,m,d},R_n) = \begin{cases}
±p^{\deg(R_n)} & \text{if }n=m\text{ and }d=p^e\text{ is a prime power}\\
±1 & \text{otherwise}.\end{cases}\]
The proof splits in several cases. 

\medskip
\noindent{\bf Case 1: $n$ does not divide $m$.}
Assume $\alpha$ is a root of $R_n$. Then, $P_{j_1}(\alpha) = P_{j_2}(\alpha)$ if and only if $j_1\equiv j_2\mod n$. 
Since $n$ does not divide $m$, for all $k\geq 2$, 
\[P_{k+m-1}(\alpha)  - P_{k-1}(\alpha)\neq 0\quad \text{and}\quad \alpha P_{k,m}(\alpha) = \frac{P_{k+m}(\alpha)-P_k(\alpha)}{P_{k+m-1}(\alpha) - P_{k-1}(\alpha)},\]
so that 
\[
\alpha^n \prod_{j=0}^{n-1} P_{k+j,m}(\alpha) = 1.
\]
The polynomial $R_n$ is monic with constant coefficient $1$. So, $\alpha$ is an algebraic unit. Thus, 
\[\prod_{j=0}^{n-1} \result(P_{k+j,m},R_n) = \prod_{j=0}^{n-1} \prod_{\alpha\in R_n^{-1}(0)} P_{k+j,m}(\alpha)  = \prod_{j=0}^{n-1} \prod_{\alpha\in R_n^{-1}(0)} \frac{1}{\alpha^n} = ±1.\]
Since $R_{k,m,d}$ divides $P_{k,m}$,  it follows that 
\[\result(R_{k,m},R_n)=±1.\]

\medskip 
\noindent{\bf Case 2: $n$ divides $m$.}
Set 
\[\nu :=\Phi_d(1,1) =  \begin{cases}
p&\text{if }d=p^e\text{ is a prime power}\\
1&\text{otherwise}.
\end{cases}\]
It is enough to prove that 
\begin{equation}\label{eq:resultgkpd}
\prod_{\ell |m} \result(R_{k,\ell,d},R_n) = ±\nu^{\deg(R_n)}.
\end{equation}
Indeed, assume Equation \eqref{eq:resultgkpd} holds. We have seen that $\result(R_{k,\ell ,d},R_n)=±1$ when $n$ does not divide $\ell$. So, for $m=n$,
\begin{align*}
 ±\nu^{\deg(R_n)} &= \result(R_{k,n,d},R_n)\cdot \prod_{\ell |n\atop \ell \neq n} \result(R_{k,\ell,d},R_n) \\
 &= ± \result(R_{k,n,d},R_n).\end{align*}
Now, if $n$ divides $m\neq n$, the polynomial $R_{k,n,d}\cdot R_{k,m,d}$ divides $P_{k,m,d}$; and
\[\result(R_{k,n,d}\cdot R_{k,m,d},R_n)=±\nu^{\deg(R_n)}\cdot \result(R_{k,m,d},R_n)\]
divides
\[\result(P_{k,m,d},R_n)  = ±\nu^{\deg(R_n)}.\]
This forces
\[\result(R_{k,m,d},R_n)=±1.\]
So, it is enough to prove that Equation \eqref{eq:resultgkpd} holds. 

\medskip 
\noindent{\bf Case 2.a:  $n$ does not divide $k-1$.}
Assume $\alpha$ is a root of $R_n$. Since $n$ divides $m$, we have that $P_{k+m-1}(\alpha) = P_{m-1}(\alpha)$ and 
\[P_{k,m,d}(\alpha) = \Phi_d\bigl(P_{k+m-1}(\alpha), P_{k-1}(\alpha) \bigr)=  P_{k-1}^{\varphi(d)}(\alpha)\cdot \Phi_d(1,1) = \nu  P_{k-1}^{\varphi(d)}(\alpha).\]
It follows that
\begin{align*}
\result(P_{k,m,d},R_n) &= \prod_{\alpha\in R_n^{-1}(0)} P_{k,m,d}(\alpha) \\
&= \nu^{\deg(R_n)} \cdot \prod_{\alpha\in R_n^{-1}(0)} P_{k-1}^{\varphi(d)}(\alpha) = \nu^{\deg(R_n)}\cdot \result(P_{k-1}^{\varphi(d)},R_n).
\end{align*}
Since $n$ does not divide $k-1$, Lemma \ref{lem:poonen} yields $\result(R_\ell,R_n)=±1$ for any divisor $\ell$ of $k-1$. Thus, 
\begin{align*}
\result(P_{k,m,d},R_n)  &= \nu^{\deg(R_n)}\cdot  \result(P_{k-1}^{\varphi(d)},R_n) \\
&=\nu^{\deg(R_n)}\cdot  \prod_{\ell|k-1} \bigl(\result(R_\ell,R_n)\bigr)^{\varphi(d)}=±\nu^{\deg(R_n)}.
\end{align*}
Equation \eqref{eq:resultgkpd} now follows from Equation \eqref{eq:fkpd}.

\medskip 
\noindent{\bf Case 2.b:  $n$ divides $k-1$.}
As in the proof of Lemma \ref{lem:poonen}, if $n$ divides $\ell$, then
\[P_\ell = P_n \mod{P_n^D} = P_n\cdot (1+H_\ell)\]
with $H_\ell\in \Z[a]$ divisible by $P_n$. It follows that 
\[P_{k,m,d} = \Phi_d\bigl(P_{k+m-1}, P_{k-1} \bigr)= P_n^{\varphi(d)}\cdot (\nu+H_{k,m,d}) \]
with $H_{k,m,d}\in \Z[a]$ divisible by $P_n$. Since $n$ divides $\gcd(m,k-1)$, Equation \eqref{eq:fkpd} yields
\[\left(\prod_{\ell|\gcd(m,k-1)\atop \ell \text{ does not divide }n} R_\ell ^{\varphi(d)}\right) \cdot \left(\prod_{\ell |m} R_{k,\ell,d} \right)= \nu +  H_{k,m,d}P_n^{D-1}\]
and since $\result(R_\ell ,R_n)=±1$ for $\ell\neq n$, we deduce that 
\begin{align*}
\prod_{\ell |m} \result(R_{k,\ell,d},R_n) &=\result(\nu+H_{k,m,d}P_n^{D-1},R_n) \\
&= \result(\nu,R_n) =  ±\nu^{\deg(R_n)}.
\end{align*}
This is Equation \eqref{eq:resultgkpd}. 

The proof of Lemma \ref{lem:resultgkmd} is completed
\end{proof}

\begin{proof}[Proof of Lemma \ref{lem:gknd}]
Assume $D=p^e$ is a prime power and $d\geq 2$ is a divisor of $D$. We need to show that for all $k\geq 2$, the polynomials $R_{k,1,d} \mod p$ are powers of $a\in \F_p[a]$; and for all $k\geq 2$ and $n\geq 2$, the polynomials $R_{k,n,d} \mod p$ are powers of $R_n\mod p$. Since  $R_{k,n,d}$ divides $R_{k,n}$ for all $n\geq 1$, it is enough to prove that for all $k\geq 2$, the polynomials $R_{k,1} \mod p$ are powers of $a\in \F_p[a]$; and for all $k\geq 2$ and $n\geq 2$, the polynomials $R_{k,n} \mod p$ are powers of $R_n\mod p$.

For $k\geq 2$, set $M_{k,1} := (D-1)N_{k-1}$ and for $n\geq 2$, set 
\[M_{k,n}:=\begin{cases}
(D-1)(D^{k-1}-1)&\text{if } n \text{ divides }k-1\\
(D-1)D^{k-1}&\text{if } n \text{ does not divide }k-1.\end{cases}\]
We shall prove that for $k\geq 2$ and $n\geq 2$,
\begin{equation}\label{eq:gknpowers}
R_{k,1} \equiv a^{M_{k,1}}\mod p \quad \text{and}\quad R_{k,n} \equiv R_n^{M_{k,n}}\mod p.
\end{equation}
Note that $N_{i+j} - N_i = D^i N_j$ for all integers $i\geq 0$ and $j\geq 0$. So, according to Lemma \ref{lem:fkfk}, if $k\geq 2$ and $n\geq 1$, 
\begin{align*}
P_{k+n-1} - P_{k-1} & \equiv a^{N_{k-1}} + a^{N_k} + \cdots + a^{N_{k+n-2}} \mod p \\
&  \equiv a^{N_{k-1}}\cdot \left(a^{D^{k-1}N_0}+a^{D^{k-1}N_1} + \cdots +a^{D^{k-1}N_{n-1}}\right)\mod p\\
&\equiv a^{N_{k-1}}\cdot \left(a^{N_0}+a^{N_1} + \cdots +a^{N_{n-1}}\right)^{D^{k-1}}\mod p \\
& \equiv a^{N_{k-1}} P_n^{D^{k-1}} \mod p.
\end{align*}
As a consequence, 
\[P_{k+n-1}^D - P_{k-1}^D \equiv a^{ DN_{k-1}} P_n^{D^k} \mod p\]
and 
\[P_{k,n} \equiv a^{(D-1)N_{k-1}} P_n^{D^k-D^{k-1}} \mod p\equiv a^{M_{k,1}} P_n^{(D-1)D^{k-1}} \mod p.\]
In particular, for $n=1$, this yields
\[R_{k,1} = P_{k,1} \equiv  a^{M_{k,1}}\mod p.\]
According to Equation \eqref{eq:buff},
\[\left(\prod_{m|\gcd(n,k-1)} R_m^{D-1}\right)\cdot \left( \prod_{m|n} R_{k,m} \right)= P_{k,n} \equiv  a^{M_{k,1}}\cdot \prod_{m|n}R_m^{(D-1)D^{k-1}}\mod p\]
and since $R_1=1$ and $R_{k,1} \equiv  a^{M_{k,1}}\mod p$, 
\[\prod_{m|n\atop m\neq 1} R_{k,m} = \prod_{m|n\atop m\neq 1}R_m^{M_{k,m}}\mod p.\]
Equation \eqref{eq:gknpowers} now follows from the Möbius inversion formula, completing the proof of Lemma \ref{lem:gknd}. 
\end{proof}

To complete the proof of Theorem \ref{theo}, we shall use the following generalization of the Eisenstein criterion.

\begin{lemma}\label{lem:eisenstein}
Assume $A\in \Z[a]$ and $B\in \Z[a]$ are monic polynomials and $p$ is a prime number such that 
\begin{itemize}
\item $A = B^N\mod p$ for some integer $N\geq 1$; 
\item the polynomial $B \mod p$ is irreducible over $\F_p$; 
\item $p^{2\deg(B)}$ does not divide $\result(A,B)$. 
\end{itemize}
Then, $A$ is irreducible over $\Q$. 
\end{lemma}

\begin{proof}
Assume by contradiction that $A$ is reducible over $\Q$, so that $A = A_1 A_2$ with $A_1\in \Z[a]$ and $A_2\in \Z[a]$ non constant. 
Let $\overline A_1$, $\overline A_2$ and $\overline B$ be the reductions of the polynomials modulo $p$. 
Then, $\overline A_1\overline A_2 = \overline B^N$ and since $\overline B$ is irreducible over $\F_p$, we have that $\overline A_1 = \overline B^{N_1}$ and $\overline A_2 = \overline B^{N_2}$ for some positive integers $N_1\geq 1$ and $N_2\geq 1$. In other words, $A_1 = B^{N_1} + pC_1$ and $A_2 = B^{N_2}+pC_2$ for some polynomials $C_1\in \Z[a]$ and $C_2\in \Z[a]$. In that case, 
\begin{align*}\result(A,B) = \result(A_1A_2,B) & = \result(A_1,B)\cdot \result(A_2,B) \\
&= \result(pC_1,B)\cdot \result(pC_2,B) \\
&= p^{2\deg(B)}\result(C_1C_2,B).\end{align*}
This contradicts the assumption that $p^{2\deg(B)}$ does not divide $\result(A,B)$. 
\end{proof}

We may now complete the proof of Theorem \ref{theo}. Assume $D=p^e$ is a prime power and $d\geq 2$ is a divisor of $D$. Then $d$ is a power of $p$. 

According to Lemma \ref{lem:gknd}, the polynomial $R_{k,1,d}\mod p$ is a power of $a\in\F_p[a]$, which is irreducible over $\F_p$; and according to Lemma \ref{lem:resultgkmd}, $p^{2\deg(R_n)}$ does not divide $\result(R_{k,1,d},R_1) = ±p^{\deg(R_n)}$. It follows from Lemma \ref{lem:eisenstein} that $R_{1,k,d}$ is irreducible over $\Q$ for all $k\geq 2$. 

Similary, according to Lemma \ref{lem:gknd}, if $n\geq 2$, the polynomial $R_{k,n,d}\mod p$ is a power of $R_n\mod p$; and according to Lemma \ref{lem:resultgkmd}, $p^{2\deg(R_n)}$ does not divide $\result(R_{k,n,d},R_n) = ±p^{\deg(R_n)}$. It follows from Lemma \ref{lem:eisenstein} that when $R_n\mod p$ is irreducible over $\F_p$, the polynomial $R_{k,n,d}$ is irreducible over $\Q$ for all $k\geq 2$. 

This completes the proof of Theorem \ref{theo}.

\subsection{Particular cases\label{app:examples}}

\noindent For small values of $k$ and $n$, the expression of $R_{k,n,d}$ is quite simple and we may obtain irreducibility as follows. 

\begin{proposition}
For all $D\geq 2$ and all $d$ that divide $D$, the polynomial $R_{2,1,d}$ is irreducible over $\Q$. 
\end{proposition}

\begin{proof}
We have that 
\[R_{2,1,d} = \Phi_d(a+1,1).\]
Since cyclotomic polynomials are irreducible over $\Q$, so is $R_{2,1,d}$.
\end{proof} 

\begin{proposition}
For all $D\geq 2$ even, the polynomial $R_{3,1,2}$ is irreducible over $\Q$. 
\end{proposition}

\begin{proof}
Setting $b:=a+1$, we have that
\[R_{3,1,2} =\Phi_2(P_3,P_2) =  P_3+P_2 = a(a+1)^D+1+(a+1) = b^{2d+1} - b^{2d}+b+1.\]
By \cite[Theorem 2]{finch}, this quadrinomial is irreducible for all $d\geq 1$. 
\end{proof}

\begin{proposition}
For all $D\geq 2$ even, the polynomial $R_{2,2,2}$ is irreducible over $\Q$. 
\end{proposition}

\begin{proof}
Assume $D=2d$ is even. Then setting $b=a+1$ as previously, 
\begin{align*}
R_{2,2,2} = \frac{\Phi_2(P_3,P_1)}{\Phi_2(P_2,P_1)} & = \frac{P_3+P_1}{P_2+P_1} \\
&= \frac{a(a+1)^D+2}{a+2} \\
&= \frac{b^{2d+1}- b^{2d}+2}{b+1} = b^{2d} - 2 b^{2d-1} + 2b^{2d-2} -\cdots -2b+2.
\end{align*}
According to the Eisenstein criterion, this polynomial is irreducible over $\Q$. 
\end{proof} 

\subsection{Irreducibility over $\F_p$\label{app:irreduc}}

\noindent Here, $D=p^e$ is a prime power. In this appendix, we shall work over the field $\F_p$ or its algebraic closure $\overline \F_p$. 
With an abuse of notation, we shall keep the notation $P_n$ and $R_n$ for their reductions modulo $p$. In other words, 
$P_n \in \F_p[a]$ and $R_n\in \F_p[a]$ are defined by 
\[P_n := \sum_{k=0}^{n-1} a^{N_k} \quad\text{with}\quad N_k := \frac{D^k-1}{D-1}
\quad \text{and}\quad R_n := \prod_{m|n} P_m^{\mu(n/m)}.\]
We study the irreducibility of $R_n$ over $\F_p$. Note that 
\[R_1 = 1\quad \text{and}\quad R_2 = a+1.\]
So, we shall restrict our study to the case $n\geq 3$. 

\begin{proposition}
Assume $D=p^e$ is a prime power and $n\geq 3$. Then, the polynomial $R_n\in \F_p[a]$ is irreducible over $\F_p$ if and only if either $n=3$ and $D=2$, or $n=3$ and $D=8$. 
\end{proposition}

\begin{proof}
Let $f:\overline\F_p\to \overline \F_p$ be the Frobenius automorphism $x\mapsto x^p$. 

\begin{lemma}
If $\alpha\in \overline \F_p$ is a root of $R_n$, then $\alpha$ is a periodic point of $f$ of period dividing $n\cdot e$.  
\end{lemma}

\begin{proof}
Assume $\alpha$ is a root of $R_n$. Then, $P_n(\alpha)=0$, so that
\begin{align*}
1 = 1+\alpha P_n^D(\alpha) &= 1+\alpha P_n(\alpha^D)  \\
& = 1 +  \sum_{k=0}^{n-1} \alpha^{1+DN_k}   \\
&=1+ \sum_{k=0}^{n-1}\alpha^{N_{k+1}}= P_n(\alpha) + \alpha^{N_n} =  \alpha^{N_n} .\end{align*}
It follows that 
\[f^{\circ (n\cdot e)}(\alpha) = \alpha^{D^n} = \alpha^{1+(D-1)N_n} = \alpha\cdot \left(\alpha^{N_n}\right)^{D-1} = \alpha.\qedhere\]
\end{proof}

As a consequence, if $R_n$ is irreducible over $\F_p$, then the degree of $R_n$ divides $n\cdot e$. The degree of $R_n$ is
\[\deg(R_n) = \sum_{m|n} \mu\left(\frac{n}{m}\right) \deg(P_m) =  \sum_{m|n} \mu\left(\frac{n}{m}\right)  N_{m-1} \geq D^{n-2}.\]
So, if $R_n$ is irreducible over $\F_p$, then $p^{(n-2)e}\leq n\cdot e.$ 

Set $\kappa:=(n-2)\log(p)>0$. The function $(0,+\infty) \ni x\mapsto {\rm exp}(\kappa x)/x\in (0,+\infty)$ reaches a minimum at $x = 1/\kappa$ with value $\kappa \cdot{\rm exp}(1)$. 
It follows that for $n\geq 3$, 
\[\frac{p^{(n-2)e}}{n\cdot e}\geq \left(1-\frac{2}{n}\right) \log (p)\exp(1).\]
If $n\geq 3$ and $p\geq 5$, or if $n\geq 4$ and $p=3$, or if $n\geq 5$ and $p=2$, this is greater than $1$. 
So, it is enough to study the following cases. 

\medskip
\noindent{\bf Case $n=3$ and $p=2$.}
In that case, for $e\geq 1$, 
\[\deg(R_n) = 1+D = 2^e+1\quad \text{and}\quad n\cdot e = 3e.\] 
The function $(0,+\infty) \ni x\mapsto (2^x+1)/(3x) \in (0,+\infty)$ is increasing on $[2,+\infty)$ and takes the values $1$ at $x=1$, $5/6$ at $x=2$ and $1$ at $x=3$. It follows that $\deg(R_n)$ divides $n\cdot e$ if and only if $e=1$ or $e=3$, i.e. $D=2$ or $D=8$; in those two cases, $R_3$ is irreducible. 

\medskip
\noindent{\bf Case $n=3$ and $p=3$.}
In that case, for $e\geq 1$, 
\[\deg(R_n) = 1+D = 3^e+1>3e=n\cdot e = 3e.\] 
So, $R_n$ cannot be irreducible in that case.

\medskip
\noindent{\bf Case $n=4$ and $p=2$.}
In that case, for $e\geq 1$, 
\[\deg(R_n) = 1+D+D^2 = 1+3^e +3^{2e}>4e=n\cdot e.\] 
So, $R_n$ cannot be irreducible in that case.
\end{proof}

\end{document}